\documentclass[preprint]{amsart}

\usepackage[english]{babel}
\usepackage[utf8]{inputenc}
\usepackage{lmodern}        
\usepackage{amsfonts,amssymb}
\usepackage{graphicx,tikz}
\usepackage{tikz-cd} 
\usepackage{float}
\usepackage{amsmath, amsthm, amsfonts}
\usepackage{hyperref}
\usepackage{enumitem}  
\usepackage[all,cmtip]{xy}
\usepackage[capitalise]{cleveref}
\usepackage{comment}
\usepackage{enumitem}
\usepackage{nicefrac}
\usepackage[official]{eurosym}
\usepackage{nomencl}
\usepackage{amsfonts}
\usepackage{hyperref}
\usepackage{tabto}
\usepackage{amssymb}
\usepackage{fancyhdr}
\usepackage{amscd}
\usepackage{xcolor}
\usepackage[english]{babel}
\usepackage{ulem}

\usepackage{tikz}
\usepackage{amssymb}
\usepackage{amsthm}
\usepackage{amsmath}
\usepackage{latexsym}
\usepackage{amsfonts}
\usepackage{color}
\usepackage{stmaryrd}
\usepackage{epic,latexsym,bezier,caption,amsbsy,psfrag,color,enumerate,amsfonts,amsmath,amscd,amssymb}
\usepackage{epsfig}
\usepackage{rotating}
\usepackage{graphics}
\usetikzlibrary{calc}
\usepackage{csquotes}       
\usepackage{amsthm}         
\usepackage{mathtools}      
\usepackage{amsfonts}       
\usepackage{amssymb}        

\usepackage{enumitem}
\usepackage{subcaption}     
\usepackage{subdepth}       
\usepackage{hyperref}       
\usepackage{tikz}           

\usetikzlibrary{automata, positioning, arrows.meta, calc, math}

\tikzstyle{graph} = [node distance = 1.5cm, every path/.style = {semithick}]
\tikzset{graph vertex/.style = {inner sep = 1pt, circle, fill}}
\tikzstyle{graph vertices} = [every node/.style = graph vertex]
\tikzset{loop up/.style = {out = 125, in = 55, looseness = 20, above}}
\tikzset{loop left/.style = {out = 145, in = 215, looseness = 20, left}}
\tikzset{loop down/.style = {out = -55, in = -125, looseness = 20, below}}
\tikzset{loop right/.style = {out = 35, in = -35, looseness = 20, right}}
\tikzset{oloop up/.style = {out = 125, in = 55, looseness = 13, above}}
\tikzset{oloop left/.style = {out = 145, in = 215, looseness = 13, left}}
\tikzset{oloop down/.style = {out = -55, in = -125, looseness = 13, below}}
\tikzset{oloop right/.style = {out = 35, in = -35, looseness = 13, right}}

\tikzstyle{automata} = [->, > = {Stealth}, node distance = 2.5cm, 
	every state/.style = {thick}, every path/.style = {semithick}]
\tikzstyle{small font} = [every node/.style = {font = \small}]
\tikzset{automata loop up/.style = {out = 105, in = 75, looseness = 7, min distance = 8mm, above}}
\tikzset{automata loop left/.style = {out = 165, in = 195, looseness = 7, min distance = 8mm, left}}
\tikzset{automata loop down/.style = {out = -75, in = -105, looseness = 7, min distance = 8mm, below}}
\tikzset{automata loop right/.style = {out = 15, in = -15, looseness = 7, min distance = 8mm, right}}

\newtheorem{thm}{Theorem}[section]
\newtheorem{prop}[thm]{Proposition}
\newtheorem{lem}[thm]{Lemma}
\newtheorem{cor}[thm]{Corollary}

\theoremstyle{definition}
\newtheorem{define}[thm]{Definition}

\theoremstyle{remark}
\newtheorem{rem}[thm]{Remark}

\numberwithin{equation}{section}    

\DeclarePairedDelimiter\br{(}{)}
\DeclarePairedDelimiter\cbr{\{}{\}}

\DeclarePairedDelimiter\abs{|}{|}

\DeclareMathOperator{\Sz}{Sz}

\makeatletter
\newcommand{\defeq}{\mathrel{\rlap{%
    \raisebox{0.3ex}{$\m@th\cdot$}}%
    \raisebox{-0.3ex}{$\m@th\cdot$}}%
=}
\makeatother

\newcommand{\N}{\mathbb{N}}

\begin{document}

\title{Topological indices on self-similar graphs generated by groups}

\author{Daniele D'Angeli}
\address{Dipartimento di Ingegneria. Universit\`{a} degli Studi Niccol\`{o} Cusano - Via Don Carlo Gnocchi, 3 00166 Roma, Italia}
\email{daniele.dangeli@unicusano.it}

\author{Stefan Hammer}
\address{TU Graz Institut f\"{u}r Diskrete Mathematik, 8010 Graz, Steyrergasse 30/III, Austria}
\email{hammer@math.tugraz.at}

\author{Emanuele Rodaro}
\address{Politecnico di Milano - P.zza Leonardo da Vinci, Milano, Italia}
\email{emanuele.rodaro@polimi.it}

\begin{abstract}
In this paper, we determine precise formulas for the diameters, the number of perfect matchings, and the Tutte polynomials for an infinite family of finite graphs, namely the Schreier graphs of tree automaton groups, also called tree graph automata. This enables us to easily find the number of spanning trees, spanning forests, and an explicit form for the chromatic polynomials. In the second part of the paper, we provide the precise values for the Wiener and Szeged index of any tree graph automaton.
\end{abstract}

\keywords{ Wiener index, Szeged index, Tutte polynomial, Tree automaton groups, Schreier graphs.}

\maketitle
\begin{center}
{\footnotesize{\bf Mathematics Subject Classification (2020)}: 05C09, 05C12, 05C30, 20E08.}
\end{center}
\section{Introduction}
The interaction between Algebra and Combinatorics is widely studied and has given rise to numerous examples of structures that have been important for the development of new research ideas. One of the most fruitful areas is the representation of a group action as a geometric object, typically a graph. Understanding which algebraic or structural properties are reflected in combinatorial and computational properties is one of the most fascinating problems studied in this field.

This paper is devoted to the study of various topological and combinatorial indices within an infinite class of self-similar graphs, specifically the Schreier graphs of tree automaton groups. Such graphs will be usually called \textit{tree graph automata} in the sequel. These graphs (actually the corresponding groups) were first introduced in \cite{articolo0} and subsequently investigated in \cite{articolo1} and \cite{articolo2}. They emerge from the action of an automaton on a rooted regular tree by automorphisms and belong to the class of Schreier graphs of automaton groups.

The study of automaton groups gained significant momentum with the Grigorchuk group, which provided the first example of a group exhibiting intermediate growth – neither polynomial nor exponential – thus addressing a significant question posed by Milnor. This class encompasses notable examples of amenable groups, Burnside groups, and groups with finite $L$-presentation. It also connects with the theory of profinite groups, combinatorics through Schreier graphs, and complex dynamics via the notion of iterated monodromy groups. Further details can be found in \cite{Handbook, nekra} and references therein.

Schreier graphs provide a representation of the action of an automaton group on the levels of a corresponding tree. They form a sequence of finite graphs that converge in the space of rooted graphs to infinite rooted graphs depicting the action on the tree's boundary. This convergence can be formalised using the notion of Gromov-Hausdorff convergence \cite{marked}. 

The limit Schreier graphs of certain tree automaton groups have been exhaustively classified in \cite{articolo1}, where it is shown that there are uncountably many isomorphism classes of infinite Schreier graphs. 

It is remarkable that tree graph automata are cactus graphs. This particular structure enables us to provide exact formulas with our computations (and not just estimations). In particular, in this paper, we derive expressions for the diameter (Proposition \ref{diametro}) and the number of perfect matchings (Theorem \ref{perfetto}) for any tree graph automata. 

The number of perfect matchings, which is linked to the dimer model and domino tilings, holds significance in statistical mechanics (see \cite{dimeri,kast, RK} and references therein). It is also worth mentioning \cite{pak} where the authors explicitly calculate the diameter of the Schreier graphs for the Aleshin and Bellaterra automaton groups.

Furthermore, we examine the Tutte polynomial of these graphs, a graph polynomial in two variables that encapsulates critical structural and combinatorial information. This polynomial's expression enables the extraction of various combinatorial indices through the specialization of its variables (refer to \cite{tutte, tutte 1} for more details and \cite{donno, donno2} for computations similar to ours). The easy factorization of the Tutte polynomial is a consequence of the inherent structure of cactus graphs (Theorem \ref{tutto}). This enables us to find a formula for the number of spanning trees, spanning forests, and the chromatic polynomial for our graphs (Corollary \ref{corollario}).

Finally, at the core of our study, we focus on the Wiener index and related indices for these graphs. Due to the difficulty of finding explicit closed formulas for infinite classes of graphs, people are also interested in extremal problems or approximations of the Wiener index (see, for instance, \cite{knor}). We surprisingly provide explicit formulas for both the Wiener index and the Szeged index of every tree graph automata.

The Wiener index was invented in 1947 by the chemist Wiener \cite{wiener}, and is used to correlate physicochemical properties with the structure of chemical compounds. It is defined as the sum of the shortest path lengths between all pairs of vertices, but can also be viewed as half the sum of the total distances of each vertex from all others (for more details, see \cite{azari, moreon, wagner}). Similarly, the Szeged index was introduced in \cite{miao} as an extension of a formula for the Wiener index of trees. The relation between such indices has been investigated by many authors (see, for instance, \cite{sonno, sandi} and references therein). In the context of cactus graphs with cycles of even length, the Wiener index can be directly derived from the Szeged index by multiplication with a factor of 2 as shown by Hammer in \cite{stefano}. 

We exploit the geometry of our graphs to give an explicit formula for the Szeged index and so derive an exact expression for the Wiener index for any tree graph automata. It is interesting that it only depends on the Wiener index of the initial tree $G$ as established in Theorem \ref{thm:TreeGraphAutomataWI}. 

We also emphasize the interplay between the Wiener index and the algebraic aspects of the underlying automaton group. Specifically, the computation of an explicit formula for the Wiener index provides a direct formula for the average distance between vertices in the associated Schreier graphs. 

The average distance in the Schreier graph reflects the typical length of geodesics connecting pairs of vertices, which has a natural algebraic interpretation. In the automaton group setting, the distance between two vertices \( p_1 \) and \( p_2 \) corresponds to the length of the shortest element \( g \) in the group that conjugates the stabilizers \( G_{p_1} \) and \( G_{p_2} \), i.e., \( g G_{p_1} g^{-1} = G_{p_2} \). This perspective ties the graph-theoretic measure of distance to the intrinsic group structure, particularly the action of the group on the set of vertices and the relationships between their stabilizers.


\section{Graph automaton groups and Schreier graphs}

In this section, we recall some basic facts about automaton groups and their Schreier graphs, focusing on graph automaton groups and the correponding tree graph automata.
\begin{define}
An \textit{automaton} or \textit{transducer} is a quadruple $\mathcal{A} =
(Q,X,\lambda,\eta)$, where:

\begin{itemize}
\item $Q$ is the set of states;
\item $X = \{1 ,2,\ldots, h\}$ is an alphabet;
\item $\lambda: Q \times X \rightarrow Q$ is the restriction map;
\item $\eta: Q \times X \rightarrow X$ is the output map.
\end{itemize}
\end{define}
The automaton $\mathcal{A}$ is \textit{finite} if $Q$ is finite and it is \textit{invertible} if, for all $q\in
Q$, the transformation $\eta(q, \cdot):X\rightarrow X$ is a permutation of $X$. An automaton $\mathcal{A}$ can be visually
represented by its \textit{Moore diagram}: this is a directed labeled graph whose vertices are identified with the states of
$\mathcal{A}$. For every state $q\in Q$ and every letter $x\in X$, the diagram has an arrow from $q$ to $\lambda(q,x)$
labeled by $x|\eta(q,x)$. A sink $id$ in $\mathcal{A}$ is a state with the property that $\lambda(id,x)=id$ and $\eta(id,x)=x$ for any $x\in X$.
For each $n\geq 1$, let $X^n$ denote the set of words of length $n$ over the alphabet $X$ and put $X^0  = \{\emptyset\}$, where $\emptyset$ is the empty word. Then the action of $\mathcal{A}$ can be naturally extended to the infinite set $X^\ast= \bigcup_{n=0}^\infty X^n$ and to the set $X^\infty = \{x_1x_2x_3\ldots \mid x_i\in X\}$ of infinite words over $X$.

For a state $q\in Q$, the transformation $\eta(q,\cdot)$ of $X^{\ast}\cup X^\infty$ is a bijection if $\mathcal{A}$ is invertible.
Given the invertible automaton $\mathcal{A}$, the \textit{automaton group} generated by $\mathcal{A}$ is by definition the group generated by the transformations $\eta(q,\cdot)$, for $q\in Q$, and it is denoted $G(\mathcal{A})$. 

The $n$-th Schreier graph $\Gamma_n=(V_{\Gamma_n}, E_{\Gamma_n})$ of the action of $G(\mathcal{A})$ on $X^n$, where $V_{\Gamma_n}$ is identified with $X^n$ and two vertices $u$ and $v$ are adjacent if there exists $q \in Q$ such that $\eta(q,u) = v$. In this case, the edge joining $u$ and $v$ is labelled by $q$.

Notice that the Schreier graph $\Gamma_n$ is a regular graph of degree $2\, \abs{Q}$ on $h^n$ vertices and it
is connected for each $n$ under the hypothesis of spherical transitivity of the action. 

In \cite{articolo0}, we introduced the following construction associating an invertible automaton with a given finite graph.

\begin{rem}
Before going on, we stress the fact that we start from an arbitrary graph and with it we are able to construct an invertible automaton that gives rise to infinitely many Schreier graphs. The graph chosen at beginning to construct the automaton and the corresponding Schreier graphs are for sure related, but in principle they are different objects.
\end{rem}

Let $G=(V,E)$ be a finite graph, where $V=\{x_1,\ldots, x_h\}$ is its vertex set and $E$ is its edge set. Now let $E'$ be the set of edges, where an orientation of each edge has been chosen. Notice that elements in $E$ are unordered pairs of type $\{x_i,x_j\}$, whereas elements in $E'$ are ordered pairs of type $(x_i,x_j)$, meaning that the edge has been oriented from the vertex $x_i$ to the vertex $x_j$.\\
\indent We then define an automaton $\mathcal{A}_G=(E' \cup \{id\}, V, \lambda, \eta)$ such that:
\begin{itemize}
\item $E' \cup \{id\}$ is the set of states;
\item $V$ is the alphabet;
\item $\lambda: E' \times V\to E'$ is the restriction map such that, for each $e=(x,y)\in E'$, one has
$$
\lambda (e,z) = \left\{
                  \begin{array}{ll}
                    e & \hbox{if } z = x, \\
                    id & \hbox{if } z \neq x;
                  \end{array}
                \right.
$$
\item $\eta: E' \times V\to V$ is the output map such that, for each $e=(x,y)\in E'$, one has
$$
\eta (e,z) = \left\{
                  \begin{array}{ll}
                    y & \hbox{if } z = x, \\
                    x & \hbox{if } z = y, \\
                    z & \hbox{if } z \neq x, y.
                  \end{array}
                \right.
$$
\end{itemize}
In other words, any directed edge $e=(x,y)$ is a state of the automaton $\mathcal{A}_G$ and it has just one restriction to itself (given by $\lambda(e,x)$) and all other restrictions to the sink $id$. Its action is nontrivial only on the letters $x$ and $y$, which are switched since $\eta(e,x)=y$ and $\eta(e,y)=x$. It is easy to check that
$\mathcal{A}_G$ is invertible for any graph $G$ and any choice of the orientation of the edges. The \textit{graph automaton group} $\mathcal{G}_G$ is defined as the automaton group generated by $\mathcal{A}_G$.
\begin{figure}[ht]
\begin{center}
\psfrag{a}{$a$}\psfrag{b}{$b$}\psfrag{1}{$1$}\psfrag{2}{$2$}\psfrag{3}{$3$}\psfrag{id}{$id$}\psfrag{1|2}{$1|2$}\psfrag{2|3}{$2|3$}\psfrag{2|1,3|3}{$2|1, 3|3$}
\psfrag{1|1,3|2}{$1|1, 3|2$}\psfrag{1|1,2|2,3|3}{$1|1, 2|2, 3|3$}
\includegraphics[width=0.7\textwidth]{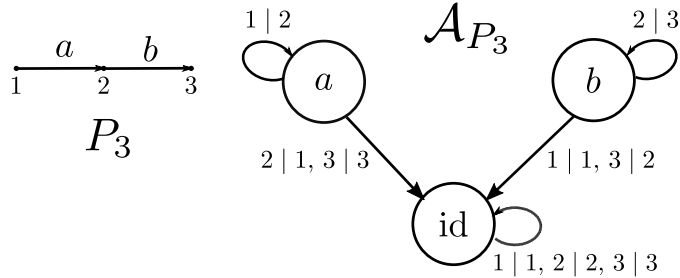}
\end{center}\caption{The path $P_3$ and the associated automaton $\mathcal{A}_{P_3}$.}
\label{automatonP3}
\end{figure}

For example, the group $G_{P_3}$ corresponds to the so-called \textit{Tangled odometer} (see Figure 1, and, e.g., \cite{articolo1}). Its two generators have the self-similar representation:
$$
a=(a,id,id)(1,2) \qquad b=(id,b,id)(2,3).
$$

\begin{figure}
    \centering
   \includegraphics[width=0.5\textwidth]{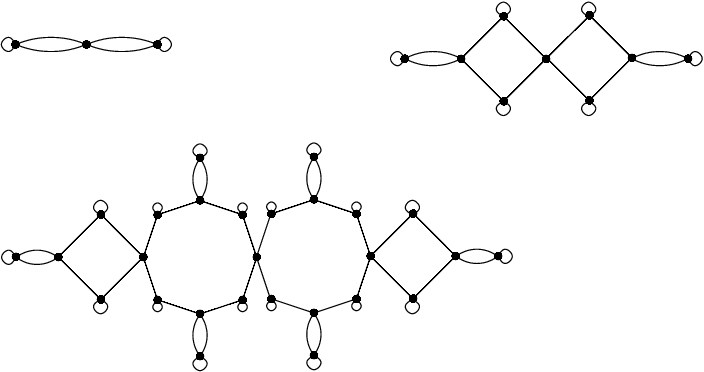}
    \caption{The Schreier graphs $\Gamma_1,\Gamma_2$ and $\Gamma_3$ of the tangled odometer.}
    \label{fig:fig_1}
\end{figure}

\begin{figure}
    \centering
   \includegraphics[width=0.5\textwidth]{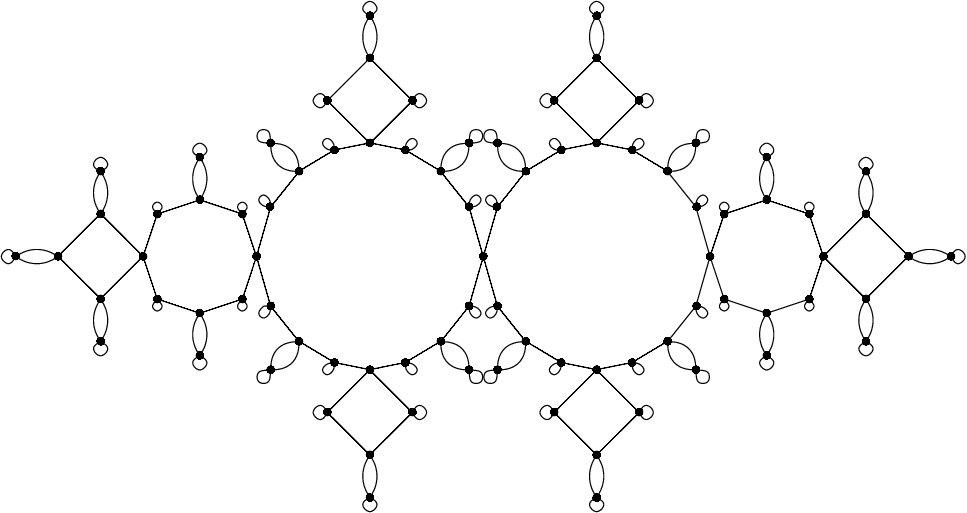}
    \caption{The Schreier graph $\Gamma_4$ of the tangled odometer.}
    \label{fig:fig_2}
\end{figure}

In this paper, we will deal with \textit{tree automaton Schreier graphs} $\Gamma^G_n$, which are the Schreier graphs of automaton groups obtained from a tree $G$ according to the construction described above. 

In what follows, given a tree $G$, we denote by $\mathcal{A}_G$ the automaton generating the associated group and by $\Gamma_n^G$ its $n-$th Schreier graph.

We will show that tree graph automata given by a tree are bipartite cactus graphs where every block is a cycle. This property is crucial to derive our main results from Theorem~\ref{thm:SzToWI}. 

Recall that a cactus is a connected graph where every two distinct cycles have at most one common
vertex. Alternatively, the graph consists of a single vertex, or every block is either an
edge or a cycle.

Let $G$ be a graph, $e$ an edge in $G$, and $C$ a cycle in $\Gamma^G_n$. If all edges of $C$ are labelled by $e$, we call $C$ an \textit{$e$-cycle}. Suppose that $G$ is a tree, then every cycle in $\Gamma^G_n$ is an $e$-cycle for some edge $e$ in $G$. 

Below, we describe many interesting and important properties of $e$-cycles. In order to make it easier to see and understand these facts, we refer to an exemplary drawing of an $e$-cycle, which can be found in Figure~\ref{fig:eCycle}. 

Let $e = (s, t)$ with vertices $s$, $t$ in $G$, and $C$ be an $e$-cycle. Then there is an $i$, $0 \leq i \leq n$, such that all vertices in $C$ coincide on the last $n - i$ positions. Furthermore, the first $i$ positions of vertices in $C$ contain only the letters $s$ and $t$, and if there is a letter in position $i + 1$, it is neither an $s$ nor a $t$. The length of $C$ is $2^i$, so the prefixes of length $i$ of the vertices are in a one-to-one correspondence with the words of length $i$ over the alphabet $\cbr{s, t}$. Every vertex has a vertex directly opposite on the cycle that only differs in position $i$. 
Thus, Schreier graphs of tree automaton groups are clearly bipartite cactus graphs where every block is a cycle. 

\begin{figure} [ht]
    \centering
    \begin{tikzpicture} [graph]
        \begin{scope} [graph vertices]
            \foreach \i in {0, ..., 15}
            {
                \node at (11.25 + 22.5 * \i : 3 cm) (\i) {};
            }
        \end{scope}
        
        \begin{scope}
            \node at (0) [right] {22113};
            \node at (1) [right] {12113};
            \node at (2) [above right] {21113};
            \node at (3) [above right] {11113};
            \node at (4) [above left] {22223};
            \node at (5) [above left] {12223};
            \node at (6) [left] {21223};
            \node at (7) [left] {11223};
            \node at (8) [left] {22123};
            \node at (9) [left] {12123};
            \node at (10) [below left] {21123};
            \node at (11) [below left] {11123};
            \node at (12) [below right] {22213};
            \node at (13) [below right] {12213};
            \node at (14) [right] {21213};
            \node at (15) [right] {11213};
            \node at (90 : 3.3 cm) {$e_C$};
            \node at (270 : 3.3 cm) {$e_C'$};
        \end{scope}
        
        \begin{scope}[every node/.style = {font = \footnotesize}]
            \foreach [evaluate = {\j = int(mod(\i + 1, 16}] \i in {0, ..., 15}
            {
                \draw (\i) edge (\j);
                \node at (\i * 22.5 : 2.7 cm) {$e$};
            }
        \end{scope}
    \end{tikzpicture}
    \caption{An $e$-cycle $C$ for $e = (1, 2)$ of size $2^4 = 16$ in the 5-th Schreier graph of the graph automaton given by a path on the three vertices 1, 2, 3, where 2 is the central vertex.}
    \label{fig:eCycle}
\end{figure}
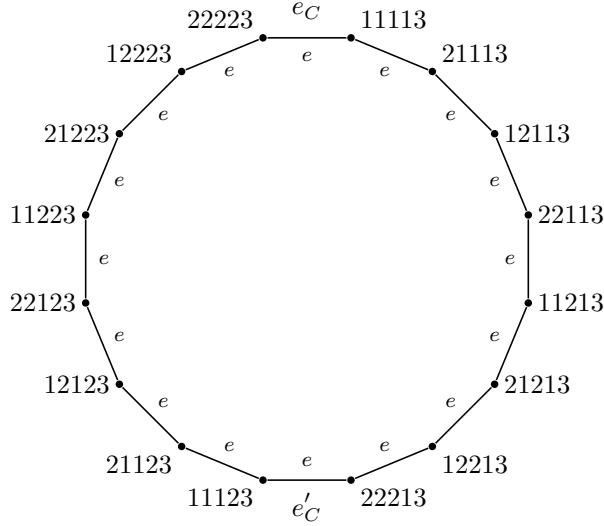

\section{Diameters and perfect matchings}

The diameter $diam(G)$ of a finite graph $G$ is defined as the maximum of the distances between any two vertices of the graph $G$.

We start by showing that the diameter of the $n$-Schreier graph of an automaton group obtained from a tree is directly given by $n$ and the diameter of the tree.

\begin{prop}\label{diametro}
 Let $d_G$ be the diameter of the tree $G$. Then, for every $n \geq 1$, 
 $$
 diam(\Gamma^G_n) = 2^{n+1} + d_G(2n-1)-4n. 
 $$
\end{prop}

\begin{proof}
Suppose that $e_1,\ldots, e_{t-1}$ is a path realizing the diameter in $G$ crossing the vertices $1,\ldots, t$. By considering the action of the generators $e_1,\ldots, e_{t-1}$ on the alphabet $\{1,\ldots, t\}^n$ we get a  a copy of the Schreier graph $\Gamma_{n}^{P_{t}}$ inside $\Gamma^G_n$. In fact the action of the generators $e_1,\ldots, e_{t-1}$ is trivial on letters different from $1,\ldots, t$. Starting from a leaf $v$ of $G$ we can consider all paths connecting it to the other leaves. Such paths give rise to sub graphs isomorphic to the Schreier graph $\Gamma_{n}^{P_{s}}$ for some $s\in \mathbb{N}$. Since the action any generator becomes trivial once it processes a vertex not belonging to it, we get that the maximal $\Gamma_{n}^{P_{s}}$ in $\Gamma_{n}^{G}$ is given exactly by $\Gamma_{n}^{P_{t}}$. Now, 
the claim follows from Theorem 4.15 in \cite{articolo0}.
\end{proof}

\begin{rem}  
Notice that the dominant exponential term in the bound of $diam(\Gamma_{n}^G)$ does not depend on $|V(G)|$.
\end{rem}   

Recall that a perfect matching in a graph is a matching that covers every vertex of the graph, i.e. given $G = (V, E)$, a perfect matching in $G$ is a subset $M$ of the edge set $E$, such that every vertex in $V$ is adjacent to exactly one edge in $M$ (see \cite{ising, dimeri}).

To study the number of perfect matching in any $\Gamma_{n}^G$, we first need the following combinatorial result.

\begin{prop} \label{prop:NumberCycles}
    Let $G$ be a tree on $k \geq 2$ vertices and $e$ an edge in $G$. Then the graph $\Gamma^G_n$ has one $e$-cycle of length $2^n$ and for $i < n$ the number of $e$-cycles of length $2^i$ is
    \begin{equation*}
        (k - 2) k^{n - i - 1} .
    \end{equation*}
\end{prop}

\begin{proof}
    Let $e = (s, t)$ with vertices $s$, $t$ in $G$. There is one $e$-cycle that contains the vertex $s^n$. This cycle has length $2^n$ and all other $e$-cycles are shorter. For fixed $i < n$, every vertex with $s^i$ as prefix and not $s$ or $t$ on position $i + 1$ is contained in a unique $e$-cycle of length $2^i$ and every such a cycle has to contain such a word. So we only need to count the words with prefix $s^i$, where in position $i + 1$ there is neither an $s$ nor a $t$. Hence, there are $k - 2$ possible letters for position $i + 1$ and $k$ possible letters for each of the remaining $n - i - 1$ positions. 
\end{proof}

Now we give an explicit formula for the number of perfect matchings in tree graph automata. Note, loops do not make sense in combination with matchings, so we can just ignore them. 

\begin{thm}\label{perfetto}
 Let $G$ be a tree on $k \geq 2$ vertices. The number of perfect matchings in $\Gamma^G_n$ is 
$$
    2^{\frac{k}{2 (k - 1)} \cdot \br*{k^n - 2 k^{n - 1} + 1}}
$$
if $G$ has a perfect matching, and it is 0 if $G$ has no perfect matching.
\end{thm}

\begin{proof}
To prove our result, we claim that $e$ is involved in the (unique) perfect matching of $G$ if and only if edges belonging to cycles labeled by $e$ are involved in every perfect matching of $\Gamma_n^G$. Since any cycle of even length only admits two perfect matchings, the number of perfect matchings is $2^{i_C}$, where $i_C$ is the number of involved cycles.

If $G$ admits a perfect matching, then $k/2$ edges are involved in it, each labelling $1+\sum_{i=1}^{n-1} (k-2)k^{n-i-1}$ cycles in $\Gamma_n^G$ (see Proposition \ref{prop:NumberCycles}). Therefore, the number of perfect matchings is
$$
2^{\frac{k}{2} \cdot \br*{1 + \sum_{i=1}^{n-1} (k-2)k^{n-i-1}}}= 2^{\frac{k}{2 (k - 1)} \cdot \br*{k^n - 2 k^{n - 1} + 1}}
$$
and the result follows.

We now prove the other claim. Consider the $n-$th Schreier graph $\Gamma_n^G$. 
If $e=(x,y)$ and $f=(y,z)$ are adjacent in $G$ then any corresponding $e-$ and $f-$cycles are connected in $\Gamma_n^G$ in a vertex of type $y^ksw$, with $s\neq y$ and $w$ a word of length $n-k-1$. Edges that do not contain $y$ in $G$ act trivially on such vertices. In this way, by using an inductive argument, we see that adjacent cycles are labeled by adjacent edges in $G$. 
In particular external cycles in $\Gamma_n^G$, i.e. those that are attached just to one other cycle once we remove loops, are labelled by extremal edges in $G$ (the ones containing a leaf). 
 Therefore, starting from extremal vertices (those of degree 2 in $\Gamma_n^G$, once we remove the loops), we can follow a sequence of cycles recovering the structure of $G$. The only way to get a perfect matching is to consider the edges containing extremal vertices.

The structure of $\Gamma_n^G$ consists of cycles that share only one vertex, so if a vertex is taken considering edges of an $e-$cycle this implies that the adjacent $f-$ cycle cannot contain edges of the perfect matching. We can proceed in this alternating manner considering the edges of adjacent cycles and since the structure of the graph repeats that of $G$, the claim follows.
\end{proof}

From the proof of the previous theorem, we get an explicit expression for the generating function. We denote by $\mathcal{M}_G$ the set of edges in the perfect matchings in the tree $G$ (when it exists).

\begin{thm}
    Let $G$ be a (labeled) tree on $k>2$ vertices that has a perfect matching. The generating function of the perfect matchings of $\Gamma_n^G$ is given by
 $$
2^{\frac{k}{2 (k - 1)} \cdot \br*{k^n - 2 k^{n - 1} + 1}}\prod_{e \in \mathcal{M}_G} e^{k^n/2}.
$$
\end{thm}
\begin{proof}
   By using the counting in the proof of Theorem \ref{perfetto} the number of perfect matchings is 
    $$
2^{\frac{k}{2 (k - 1)} \cdot \br*{k^n - 2 k^{n - 1} + 1}}
 $$
 If $e$ appears in the perfect matching of $G$, then the cycles labelled by $e$ of length $l$ involve $l/2$ edges in a perfect matching of $\Gamma_n^G$. So, by using Proposition \ref{prop:NumberCycles}, we know that there is one $e-$cycle of length $2^n$ and $ (k - 2) k^{n - i - 1} $ $e-$cycle of length  $2^i$, for $i=1, \ldots, n-1$. Since only half of such $e-$dges are involved in a perfect matching, we get that such number is
 $$
 2^{n-1}+ \sum_{i=0}^{n-1} (k-2)k^i2^{n-i-2}= 2^{n-1}+ (k-2)2^{n-2}\sum_{i=0}^{n-1} \left(\frac{k}{2}\right)^i.
 $$
 The result follows by observing that
 $$
\sum_{i=0}^{n-1} \left(\frac{k}{2}\right)^i=\frac{2^n-k^n}{2^{n-1}(2-k)}.
 $$
 
    \begin{equation*}
    \end{equation*}
\end{proof}

\section{The Tutte polynomial}

In this section, we give an explicit expression of the Tutte polynomial $T(\Gamma_n^G, x,y)$ for any tree graph automata. 
We recall that the Tutte polynomial of a graph $G$ is a 2-variable polynomial 
recursively defined as follows:
$$
T(G,x,y)=\begin{cases}
1 &\textrm{if $G$ is a vertex,} \\
x T(G\setminus e, x,y) &\textrm{if $e$ is a bridge,} \\
y T(G\setminus e, x,y) &\textrm{if $e$ is a loop,} \\
T(G\setminus e, x,y)+ T(G/ e, x,y) &\textrm{otherwise,}
\end{cases}
$$

where $G\setminus e$ and $G/e$ represent the deletion and the contraction of an edge $e$, respectively.

The Tutte polynomial can be evaluated at particular numbers $(x,y)$ to give
numerical graphical invariants, including the number of spanning trees, the number
of forests, the number of connected spanning subgraphs, the dimension of the bicycle
space, etc. (see, for instance, \cite{tutto} and references therein).

\begin{thm}\label{tutto}
Let $G$ be a tree on $k \geq 2$ vertices. The Tutte polynomial of $\Gamma_n^G$ is given by
\begin{equation*}
\begin{split}
    T(\Gamma_n^G, x,y) &= y^{(k-1)(k-2)k^{n-1}}\cdot \left(y+x+\cdots+ x^{2^n-1}\right)^{k-1} \\
        &\quad \cdot \prod_{i=2}^{n-1}\left(y+x+\cdots+ x^{2^i-1}\right)^{(k-1)(k-2)k^{n-i-1}}
\end{split}
\end{equation*}
\end{thm}
\begin{proof}
It is well known (see for instance \cite{tutte}) that the Tutte polynomial of a cactus graph made by cycles is the product of the Tutte polynomials of the composing cycles. The Tutte polynomial of a cycle of length $m$ is given by 
\begin{equation*}
    y + x + \cdots + x^{m-1}
\end{equation*}
(see, for instance, \cite[Lemma 2.4]{donno2}). The theorem follows by applying Proposition~\ref{prop:NumberCycles}.   
\end{proof}
We deduce a series of corollaries.

\begin{cor}\label{corollario}
Let $G$ be a tree on $k \geq 2$ vertices and consider the graph $\Gamma_n^G$. 
\begin{itemize}
    \item The number of spanning trees in $\Gamma_n^G$ is 
    $$
    T(\Gamma_n^G,1,1)= 2^{n +\frac{\br*{k^{n - 2}(2k - 1) - 1} (k - 2)}{k - 1}} .
    $$

\item The number of spanning forests in $\Gamma_n^G$ is 
    $$
    T(\Gamma_n^G,2,1)= \left(2^{2^n}-1\right)^{k-1} \prod_{i=2}^{n-1} \left(2^{2^i}-1\right)^{(k-1)(k-2)k^{n-i-1}} .
    $$
    
\item The chromatic polynomial $\chi(\overline{\Gamma_n^G},\lambda)$ of $\Gamma_n^G$ without loops is given by 
    \begin{equation*}
        \begin{split}
            (-1)^{k+n}\lambda T(\overline{\Gamma_n^G},1-\lambda, 0)&= (-1)^{k+1} \left((1-\lambda)^{2^n}-\lambda+1\right)^{k-1} \\
                &\quad \cdot \prod_{i=2}^{n-1} \left((1-\lambda)^{2^i}-\lambda+1\right)^{(k-1)(k-2)k^{n-i-1}} .
        \end{split}
    \end{equation*}
\end{itemize}
\end{cor}
\begin{proof}
   For the first statement, observe that, by using Theorem \ref{tutto}, each factor of the product contributes with $(2^i)^{(k-1)(k-2)k^{n-i-1}}$. Thus,
    \begin{equation*}
        \begin{split}
        T(\Gamma_n^G,1,1) &= 2^{n(k-1)}\cdot \prod_{i=2}^{n-1}(2^i)^{(k-1)(k-2)k^{n-i-1}}\\
            &= 2^{n(k-1)}\cdot \left(2^{\sum_{i=2}^{n-1}i/k^i}\right)^{(k-1)(k-2)k^{n-1}}\\
            &= 2^{n(k-1)}\cdot 2^{(2-k)n-\frac{k-2}{k-1}+\frac{k^{n-2}(2k-1)(k-2)}{k-1}}\\
            &= 2^{n +\frac{\br*{k^{n - 2}(2k - 1) - 1} (k - 2)}{k - 1}}, 
        \end{split}
    \end{equation*}
   where  
   $$
   \sum_{i=1}^n i x^i = x\frac{1-x^n}{(1-x)^2} - \frac{n x^{n+1}}{1-x}
   $$
   is used with $x = 1 / k$, which is true for $x \neq 1$.

   For the second equality, observe that inside any bracket of the Tutte polynomial, we get
   $$
   \sum_{j=0}^{2^i-1}2^j=2^{2^i}-1.
   $$
   
   For the third formula, it refers to the graph without loops, so we are deleting the factor $y^{(k-1)(k-2)k^{n-1}}$ from the polynomial.  Observe that inside any bracket of the Tutte polynomial, we get
   $$
   \sum_{j=1}^{2^i-1}(1-\lambda)^j=\dfrac{(1-\lambda)^{2^i}-(1-\lambda)}{-\lambda}.
   $$
  
   Since there are $n$ factors in the decomposition of the Tutte polynomial, we get the assert.
\end{proof}

\section{Wiener and Szeged Index of Tree Graph Automata}

We start this section by recalling the following definitions.

The \textit{Wiener index} $W(G)$ is the sum of all distances in a graph, formally
\begin{equation*}
    W(G) \defeq \sum_{u, v \in V(G)} d_G(u, v) = \frac{1}{2} \sum_{v \in V(G)} d_G^v(v) .
\end{equation*}
where $d_G^v(v)$ i the total-distance from $v$ (i.e. the sum of distances of all vertices in $G$ from $v$).

The \textit{Szeged index} of $G$ is
\begin{equation*}
    \Sz(G) \defeq \sum_{\substack{e \in E(G) \\ e = \{u, v\}}} n(u, v) n(v, u) ,
\end{equation*}
where $n(u, v)$ denotes the number of vertices in $G$ closer to $u$ than to $v$. 

The main objective of this chapter is to show the following theorem. 

\begin{thm} \label{thm:TreeGraphAutomataWI}
    Let $G$ be a tree on $k$ vertices, then
    \begin{equation*}
        \begin{split}
            W\br*{\Gamma^G_n} &= 
                \frac{(k - 1)^2 (2 k^2 - 2 k - 1)}{k^3 (2 k - 1)} \cdot 2^n k^{2 n} 
                    - \frac{2 (k - 1)^2}{k^2} \cdot n k^{2 n} \\
                &\quad+ \frac{2 (k + 1) (k - 1)}{k^3} \cdot k^{2 n} 
                    - \frac{2 (k + 1) (k - 1)}{k (2 k - 1)} \cdot k^n \\
                &\quad+ \br*{\frac{2}{k^2} \cdot n k^{2 n} 
                        - \frac{(k^2 + 2)}{k^3 (k - 1)} \cdot k^{2 n}
                        + \frac{(k + 2)}{k^2 (k - 1)} \cdot k^n}
                    W(G) .
        \end{split}
    \end{equation*}
\end{thm}

In essence, this theorem tells us that the Wiener index of the $n$-th Schreier graph of a graph automaton given by a tree $G$ depends only on $n$, and by the size $k$ and the Wiener index of $G$. Furthermore, the term $2^n k^{2 n}$ of highest order is independent of the structure of the tree. 

The most important ingredient in the proof is the following theorem. 

\begin{thm}[\cite{stefano}, Theorem~3.1] \label{thm:SzToWI}
    Let $G$ be a cactus graph, then
    \begin{equation*}
        \Sz(G) \leq 2 \, W(G), 
    \end{equation*}
    with equality if and only if every block of $G$ is a cycle of even length. 
\end{thm}
Since loops change nothing for the calculation of the Wiener and the Szeged index, the following is a consequence of Theorem~\ref{thm:SzToWI}.  

\begin{cor} \label{cor:SzToWI}
    Let $G$ be a tree, then for all $n \in \N$, 
    \begin{equation*}
        \Sz\br*{\Gamma^G_n} = 2 \, W\br*{\Gamma^G_n} . 
    \end{equation*}
\end{cor}

To derive the Szeged index, we need to evaluate the contribution of $n(u, v) n(v, u)$ for each edge $\cbr{u, v}$ in $\Gamma^G_n$. In order to simplify the calculations, we classify the edges using $e$-cycles. Let $C$ be an $e$-cycle of length $2^i$ and $w$ the common suffix of all vertices in $C$. It turns out, that all edges except the two special edges $e_C$ connecting $s^i w$ and $t^i w$ and $e_C'$ connecting $s^{i - 1} t w$ and $t^{i - 1} s w$ in $C$ have the same contribution, and this contribution is practically given by $e$ and the size of $C$. As shown in Figure~\ref{fig:eCycle}, the special edges $e_C$ and $e_C'$ are directly opposite to each other on the cycle. Since we already know the number of $e$-cycles by Proposition~\ref{prop:NumberCycles}, the next step is to determine the contribution of non-special edges. This requires some more structural insights. 

There may be one vertex whose deletion disconnects the graph and every obtained component has the same number of vertices, but for all other vertices $u$ the deletion of $u$ results in a graph that has a unique largest component. For each of these vertices $u$, we define its \textit{decoration} to be the graph obtained by taking $\Gamma^G_n$ and deleting all vertices contained in the unique largest component resulting from the deletion of $u$. This concept coincides with the definition of decorations given in~\cite{articolo0}. 

\begin{lem} \label{lem:NonSpecialContribution}
    Let $G$ be a tree on $k$ vertices and $\cbr{u, v}$ a non-special edge in an $e$-cycle of length $2^i$. Then $\cbr{u, v}$ satisfies
    \begin{equation*}
        n(u, v) n(v, u) = k^{i - 1} \br*{k^n - k^{i - 1}} .
    \end{equation*}
\end{lem}

\begin{proof}
    Let $C$ be the $e$-cycle and $e = (s, t)$. Then $u$ and $v$ have vertices $u'$ and $v'$ directly opposite to them on $C$ that differ only in the $i$-th position. All vertices closer to $u$ than to $v$ are also closer to $v'$ than to $u'$, so removing $\cbr{u, v}$ and $\cbr{u', v'}$ splits the graph into two components, where one has $n(u, v)$ and the other $n(v, u)$ vertices. 
    
    Suppose without loss of generality that the component $K_u$ with $n(u, v)$ vertices does not have more vertices than the other component. Then, since $\cbr{u, v}$ is non-special, all vertices in $K_u$ are of the form $x s w$ or $x t w$, where $x$ is a word of length $i - 1$, and $w$ is the common suffix of the vertices in $C$. 
    
    Furthermore, if $x s w$ is in $K_u$, it is in the decoration of some vertex in $C$, so $x t w$ is in the decoration of the opposite vertex in $C$, and hence not in $K_u$. On the other hand, if $x s w$ is not in $K_u$, it has to be in the decoration of some vertex on the cycle, and thus $x t w$ is in $K_u$. 
    
    In conclusion, this shows that taking the prefixes of length $i - 1$ of $K_u$ gives a bijection to the words of length $i - 1$. Hence, 
    \begin{equation*}
        n(u, v) = k^{i - 1}, \quad \text{and} \quad n(v, u) = k^n - k^{i - 1}
    \end{equation*}
    follows by bipartivity. 
\end{proof}
The more complex case concerning the contribution of special edges is addressed by the following two results. 

\begin{lem} \label{lem:ContributionSpecialC2n}
    Let $G$ be a tree on $k$ vertices, $e = (s, t)$ an edge in $G$, and $\cbr{u, v}$ a special edge in the only $e$-cycle of length $k^n$. Then, 
    \begin{equation*}
        \begin{split}
            n(u, v) n(v, u) &= n(s, t) n(t, s) k^{2 (n - 1)} \\
                &= n(s, t) k^{n - 1} \br*{k^n - n(s, t) k^{n - 1}} \\
                &= n(t, s) k^{n - 1} \br*{k^n - n(t, s) k^{n - 1}} .
        \end{split}
    \end{equation*}
\end{lem}

\begin{proof}
    Without loss of generality suppose that $u = s^n$ and $v = t^n$. Removing $(u, v)$ and $(u', v')$, where $u' = s^{n - 1} t$ and $v' = t^{n - 1} s$, splits the graph into two components. Let $K_u$ be the component that contains $u$. Then $K_u$ contains $n(u, v)$ vertices, and every vertex with $s$ in position $n$ is contained in $K_u$. Moreover, every vertex in $K_u$ can only have a symbol $s'$ in position $n$ if $s'$ is closer to $s$ than to $t$ in $G$. The same argument can be applied to the other component with the roles of $s$ and $t$ exchanged. Hence, 
    \begin{equation*}
        n(u, v) = \abs{K_u} = k^{n - 1} n(s, t)
            \quad \text{and} \quad
        n(v, u) = k^{n - 1} n(t, s) .
    \end{equation*}
    This proves the first equality. The second and third follow by
    \begin{equation*}
        k = n(s, t) + n(t, s) ,
    \end{equation*}
    which follows from the bipartivity of the tree $G$. 
\end{proof}

\begin{lem} \label{lem:ContributionSpecial}
    Let $G$ be a tree on $k$ vertices, $e = (s, t)$ an edge in $G$, and $i < n$. In $(n(t, s) - 1) k^{n - 1 - i}$ $e$-cycles of length $2^i$ each of the two special edges contributes
    \begin{equation*}
        n(s, t) k^{i - 1} \br*{k^n - n(s, t) k^{i - 1}}
    \end{equation*}
    and in the remaining $(n(s, t) - 1) k^{n - 1 - i}$ $e$-cycles of length $2^i$ each of the two special edges contributes
    \begin{equation*}
        n(t, s) k^{i - 1} \br*{k^n - n(t, s) k^{i - 1}}
    \end{equation*}
    to the Szeged index. 
\end{lem}

\begin{proof}
    Take the same decomposition as in the proof of the last lemma, again with $K_u$ as the component that contains $u$. Without loss of generality suppose that $u = s^i w$, where $w$ is the common prefix of all vertices on the $e$-cycle. If the letter on the first position of $w$ is closer to $t$ than to $s$ in $G$, it cannot be changed in $K_u$ and the same is true for all letters after position $i$. Thus, taking prefixes of length $i$ gives a graph isomorphism to the component of $s^i$ in $\Gamma^G_i$. Hence, by Lemma~\ref{lem:ContributionSpecialC2n}, 
    \begin{equation*}
        \begin{split}
            n(u, v) = \abs{K_u} = n(s, t) k^{i - 1} ,
            \quad \text{and by bipartivity, } \quad 
            n(v, u) = k^n - n(s, t) k^{i - 1} .
        \end{split}
    \end{equation*}
    There are $(n(t, s) - 1) k^{n - 1 - i}$ distinct words of type $s^i w$, where the letter on the first position of $w$ is closer to $t$ than to $s$ in $G$ and all those take the role of $u$ once. So the first assertion is true, and exchanging the roles of $s$ and $t$ the second follows. 
\end{proof}

Now we have a complete classification of the edges including their respective contributions to the Szeged index. 

\begin{thm} \label{thm:TreeGraphAutomataSz}
    Let $G$ be a tree on $k$ vertices and $n \in \N$, then
    \begin{equation*}
        \begin{split}
            \Sz\br*{\Gamma^G_n} &= 
                \frac{2 (k - 1)^2 (2 k^2 - 2 k - 1)}{k^3 (2 k - 1)} \cdot 2^n k^{2 n} 
                    - \frac{4 (k - 1)^2}{k^2} \cdot n k^{2 n} \\
                &\quad+ \frac{4 (k + 1) (k - 1)}{k^3} \cdot k^{2 n} 
                    - \frac{4 (k + 1) (k - 1)}{k (2 k - 1)} \cdot k^n \\
                &\quad+ \br*{\frac{4}{k^2} \cdot n k^{2 n} 
                        - \frac{2 (k^2 + 2)}{k^3 (k - 1)} \cdot k^{2 n}
                        + \frac{2 (k + 2)}{k^2 (k - 1)} \cdot k^n}
                    \Sz(G) .
        \end{split}
    \end{equation*}
\end{thm}

\begin{proof}
    To make the calculations easier to follow, we split it up into four parts and use $E = E(G)$ for the edge set of the tree $G$. By Proposition~\ref{prop:NumberCycles} and Lemma~\ref{lem:NonSpecialContribution}, the contribution of non-special edges in cycles of length $2^n$ is
    \begin{equation*}
        A = \sum_{e \in E} \br*{2^n - 2} \cdot k^{n - 1} \br*{k^n - k^{n - 1}} 
            = (k - 1)^2 k^{2 n - 2} \br*{2^n - 2} .
    \end{equation*}
    The contribution of special edges in cycles of length $2^n$ follows from Proposition~\ref{prop:NumberCycles} and Lemma~\ref{lem:ContributionSpecialC2n}. 
    \begin{equation*}
        B = \sum_{(s, t) \in E} 2 \cdot n(s, t) n(t, s) k^{2 (n - 1)}
            = 2 k^{2 n - 2} \Sz(G) .
    \end{equation*}
    Once more Proposition~\ref{prop:NumberCycles} and Lemma~\ref{lem:NonSpecialContribution} are needed for the contribution of non-special edges in cycles of length $2^i$ for $i < n$. 
    \begin{equation*}
        \begin{split}
            C_i &= \sum_{e \in E} (k - 2) k^{n - 1 -i} \cdot \br*{2^i - 2} 
                    \cdot k^{i - 1} \br*{k^n - k^{i - 1}} \\
                &= (k - 2) (k - 1) k^{n - 2} \br*{2^i - 2} \br*{k^n - k^{i - 1}} .
        \end{split}
    \end{equation*}
    These contributions $C_i$ have to be considered for all $i$ with $1 \leq i < n$.
    \begin{equation*}
        \begin{split}
            C &= \sum_{i = 1}^{n - 1} C_i
                = (k - 2) (k - 1) k^{n - 2} 
                    \sum_{i = 1}^{n - 1} \br*{k^n 2^i - 2 \cdot (2 k)^{i - 1} - 2 k^n + 2 k^{i - 1}} \\
                &= (k - 2) (k - 1) k^{n - 2} 
                    \left(\rule{0cm}{0.8cm}\right. 
                        k^n \br*{2^n - 2} 
                        - \frac{2 \br*{(2 k)^{n - 1} - 1}}{2 k - 1} \\
                    &\phantom{= (k - 2) (k - 1) k^{n - 2}}
                        - 2 (n - 1) k^n
                        + \frac{2 \br*{k^{n - 1} - 1}}{k - 1} 
                    \left.\rule{0cm}{0.8cm}\right) .
        \end{split}
    \end{equation*}
    The final contribution comes from special edges in cycles of length $2^i$ for $i < n$, and is given by Lemma~\ref{lem:ContributionSpecial}. 
    \begin{equation*}
        \begin{split}
            D_i &= \sum_{(s, t) \in E} 
                    \Bigg( (n(t, s) - 1) k^{n - 1 - i} \cdot 2 
                        \cdot n(s, t) k^{i - 1} \br*{k^n - n(s, t) k^{i - 1}} \\
                    &\phantom{= \sum_{(s, t) \in E}} 
                        + (n(s, t) - 1) k^{n - 1 - i} \cdot 2 
                        \cdot n(t, s) k^{i - 1} \br*{k^n - n(t, s) k^{i - 1}} \Bigg) \\
                &= 2 k^{n - 2} \sum_{(s, t) \in E} 
                    \Bigg( 2 k^n n(s, t) n(t, s) 
                        - n(s, t) n(t, s) \big( n(s, t) + n(t, s) \big) k^{i - 1} \\
                    &\phantom{= 2 k^{n - 2} \sum_{(s, t) \in E}} 
                        - \big( n(s, t) + n(t, s) \big) k^n 
                        + \big( n(s, t)^2 + n(t, s)^2 \big) k^{i - 1} \Bigg) .
        \end{split}
    \end{equation*}
    Since $n(s, t) + n(t, s) = k$ and
    \begin{equation*}
        n(s, t)^2 + n(t, s)^2 = \big( n(s, t) + n(t, s) \big)^2 - 2 n(s, t) n(t, s) = k^2 - 2 n(s, t) n(t, s) ,
    \end{equation*}
    $D_i$ can be simplified in the following way. 
    \begin{equation*}
        \begin{split}
            D_i &= 2 k^{n - 2} \sum_{(s, t) \in E} 
                    \Bigg(2 k^n n(s, t) n(t, s) 
                        - n(s, t) n(t, s) k^i \\
                    &\phantom{= 2 k^{n - 2} \sum_{(s, t) \in E}} - k^{n + 1}
                        + k^{i + 1} - 2 n(s, t) n(t, s) k^{i - 1} \Bigg) \\
                &= 2 k^{n - 2} \br*{2 \Sz(G) k^n - (k + 2) \Sz(G) k^{i - 1} 
                    - (k - 1) k^{n + 1} + (k - 1) k^{i + 1} }.
        \end{split}
    \end{equation*}
    As before, $D_i$ has to be considered for all $i$ with $1 \leq i < n$. 
    \begin{equation*}
        \begin{split}
            D &= \sum_{i = 1}^{n - 1} D_i \\
                &= 2 k^{n - 2}
                    \sum_{i = 1}^{n - 1} \br*{2 \Sz(G) k^n - (k + 2) \Sz(G) k^{i - 1} 
                        - (k - 1) k^{n + 1} + (k - 1) k^{i + 1} } \\
                &= 2 k^{n - 2} 
                    \left(\rule{0cm}{1cm}\right. 
                        \br*{\rule{0cm}{0.8cm}
                        2 (n - 1) k^n 
                        - \frac{(k + 2) \br*{k^{n - 1} - 1}}{k - 1} } \Sz(G) \\
                    &\phantom{2 k^{n - 2}} 
                        - (n - 1) (k - 1) k^{n + 1}
                        + k^2 \br*{k^{n - 1} - 1} 
                    \left.\rule{0cm}{1cm}\right) .
        \end{split}
    \end{equation*}
    Now, plugging all these expressions into  
    \begin{equation*}
        \Sz\br*{\Gamma^G_n} = A + B + C + D ,
    \end{equation*}
    the formula for $\Sz\br*{\Gamma^G_n}$ follows with some more standard reformulations. 
\end{proof}

\begin{proof}[Proof of Theorem~\ref{thm:TreeGraphAutomataWI}]
    The Wiener index and the Szeged index coincide on trees. Thus, the claim follows directly by Theorem~\ref{thm:TreeGraphAutomataSz} and Corollary\ref{cor:SzToWI}. 
\end{proof}

On trees the Wiener index is maximised by paths and minimised by stars. Due to Theorem~\ref{thm:TreeGraphAutomataSz}, this is also true for the Wiener index of Schreier graphs of graph automata that arise from trees. Below, we present these maxima and minima. 

\begin{cor}
    Let $G$ be a path on $k$ vertices, then
    \begin{equation*}
        \begin{split}
            W\br*{\Gamma^G_n} &= 
                \frac{(k - 1)^2 (2 k^2 - 2 k - 1)}{k^3 (2 k - 1)} \cdot 2^n k^{2 n} 
                    + \frac{(k - 1) (k - 2) (k - 3)}{3 k^2} \cdot n k^{2 n} \\
                &\quad- \frac{(k + 1) (k - 2) (k^2 + 2 k - 6)}{6 k^3} \cdot k^{2 n} 
                    + \frac{(k + 1) (k - 2) (2 k - 5)}{6 k (2 k - 1)} \cdot k^n .
        \end{split}
    \end{equation*}
\end{cor}

\begin{proof}
    The Wiener index of $G$ is
    \begin{equation*}
        W(G) = \frac{k (k^2 - 1)}{6} .
    \end{equation*}
\end{proof}

\begin{cor}
    Let $G$ be a star on $k$ vertices, then
    \begin{equation*}
        W\br*{\Gamma^G_n} = 
            \frac{(k - 1)^2 (2 k^2 - 2 k - 1)}{k^3 (2 k - 1)} \cdot 2^n k^{2 n} 
            - \frac{(k - 1) (k - 2)}{k^2} \cdot k^{2 n}
            + \frac{(k - 1) (k - 2)}{k^2 (2 k - 1)} \cdot k^n .
    \end{equation*}
\end{cor}

\begin{proof}
    The Wiener index of $G$ is
    \begin{equation*}
        W(G) = (k - 1)^2 .
    \end{equation*}
\end{proof}

\begin{rem}
Proposition \ref{diametro} says that, independently from the shape of $G$, the dominating term of the diameter of $\Gamma_n^G$ is $2^{n+1}$. Moreover, the number of vertex pairs in $\Gamma_n^G$ is $k^{2n}/2$. Theorem \ref{thm:TreeGraphAutomataWI} guarantees that the dominating term of the Wiener index of $\Gamma_n^G$ equals 
\begin{equation*}
    W(G)\frac{(k - 1)^2 (2 k^2 - 2 k - 1)}{k^3 (2 k - 1)}2^nk^{2n} .    
\end{equation*}
This implies that the asymptotic ratio between the Wiener index and the diameter times the number of pairs is the constant
$$
\frac{W(G)}{2}\frac{(k - 1)^2 (2 k^2 - 2 k - 1)}{k^3 (2 k - 1)} .
$$
We leave open the problem of whether this always occurs in the context of Schreier graphs of automaton groups.
\end{rem}

\section{Acknowledgements}

Stefan Hammer acknowledges the support of the Austrian Science Fund: FWF-P31889-N35. D. D'Angeli and E. Rodaro are members of the National Research Group GNSAGA (Gruppo Nazionale per le Strutture Algebriche, Geometriche e le loro Applicazioni) of Indam.

\end{document}